\documentclass[psamsfonts]{amsart}
\usepackage{amssymb,amsfonts}
\usepackage{enumerate}
\usepackage{mathrsfs}
\usepackage{url}
\usepackage{graphicx}
\usepackage{float}

\newtheorem{thm}{Theorem}[section]

\newtheorem{prop}[thm]{Proposition}

\theoremstyle{definition}

\newtheorem{question}[thm]{Question}

\theoremstyle{remark}

\makeatletter
\let\c@equation\c@thm
\makeatother
\numberwithin{equation}{section}

\title{Computation of a Definite Integral}
\author{Juan Arias de Reyna}
\address{Facultad de Matem\'aticas \\
Univ.~de Sevilla \\
Apdo.~1160
 \\
41080-Sevilla \\
Spain} 
\thanks{Supported by  MINECO grant MTM2012-30748.}
\email{arias@us.es}
\date{\today}

\begin{document}

\newcommand{\N}{{\mathbb N}}
\newcommand{\R}{{\mathbb R}}
\newcommand{\C}{{\mathbb C}}
\newcommand{\Z}{{\mathbb Z}}
\newcommand{\Q}{{\mathbb Q}}
\newcommand{\arctanh}{\mathop{\rm arctanh }}
\def\Re{\operatorname{Re}}
\def\Im{\operatorname{Im}}
\newfont{\cmbsy}{cmbsy10}
\newcommand{\Orden}{\mathop{\hbox{\cmbsy O}}\nolimits}

\begin{abstract}
As an  application  of Cauchy's Theorem we prove that 
\[\int_0^1\arctan\Bigl(\frac{\arctanh x-\arctan x}{\pi+\arctanh x-\arctan x}\Bigr)
\frac{dx}{x}= \frac{\pi}{8}\log\frac{\pi^2}{8}\]
answering a question first posted in Mathematics  Stack Exchange and then in MathOverflow.
\end{abstract}

\maketitle

\section{Introduction}

In 2013, on  August 11  a user with nickname  ``larry'' \cite{larry} posted the following question 
in \emph{Mathematics Stack Exchange} 
\begin{question}
How to prove that
\begin{equation}\label{question}
\int_0^1\arctan\Bigl(\frac{\arctanh x-\arctan x}{\pi+\arctanh x-\arctan x}\Bigr)
\frac{dx}{x}= \frac{\pi}{8}\log\frac{\pi^2}{8}?
\end{equation}
\end{question}

Mathematics Stack Exchange is a question and answer site for people 
studying math at any level, and professionals in related fields. Many times
we see there questions of homework level. The above question received many comments
but not a genuine answer. 

Later, in 2014, on the 18th of January, the same question was also posted in 
MathOverflow \cite{zy},
which is a question and answer site for professional 
mathematicians.
The question in MathOverflow was posted by a user with nickname ``zy\_''.
Again, until now, we have not found a proof of \eqref{question} in MathOverflow.  

In Mathematics  Stack  Exchange ( October 10, 2013 ) we also find a related 
integral posted by Vladimir 
Reshetnikov \cite{Reshetnikov}.
This received two answers but no solution to 
question \eqref{question}. 

Although some of the comments ask for the origin of the value of the 
integral, this question has not been answered by larry. 
Given the powerful methods existing today to guess such values 
\cite{exp}, there is a possibility that the given value was just an educated guess.
Therefore, we think a proof of \eqref{question} deserves to be published.  It is 
a good example of complex integration, multivalued functions, 
and Cauchy's Theorem.

\section{Preparations.}

\begin{prop}
The integral in \eqref{question} is well defined.
\end{prop}

\begin{proof}
For $-1<x<1$ we have
\[\arctanh x=\int_0^x \frac{dt}{1-t^2},\quad \arctan x=\int_0^x\frac{dt}{1+t^2}\]
so that 
\begin{equation}\label{def f}
\arctanh x-\arctan x=\int_0^x \frac{2t^2}{1-t^4}\,dt.
\end{equation}
Therefore,
\begin{equation}\label{gooddef}
f(x):=\frac{1}{\pi}(\arctanh x-\arctan x)
\end{equation}
is a  differentiable, strictly increasing and non negative function on $[0,1)$. 
It follows that $\frac{f(x)}{1+f(x)}$ is continuous. Also $f(0)=0$ so that 
\[\arctan\Bigl(\frac{\arctanh x-\arctan x}{\pi+\arctanh x-\arctan x}\Bigr)
\frac{1}{x}=\arctan\Bigl(\frac{f(x)}{1+f(x)}\Bigr)\frac{1}{x}\]
is continuous and bounded in $[0,1)$. 
\end{proof}

In this paper  $\log z$ always denotes the main branch of the logarithmic 
function defined by $\log z=\log|z|+i\arg(z)$, with $|\arg(z)|<\pi$. This function 
is analytic in the complex plane with a cut along the  negative real axis. 

Let $\Omega\subset\C$ be the complex plane with four cuts, two along the real axis, one
from $1$ to $+\infty$, the other from $-1$ to $-\infty$, and  two along the imaginary 
axis, one from $i$ to $+i\infty$ the other from $-i$ to $-i\infty$. 
This is a star-shaped open set with center at $0$.

\begin{prop}
The function $f(x)$ defined in \eqref{gooddef} extends to an analytic function 
on the simply connected open set $\Omega$ and we have 
\begin{equation}\label{logform}
f(z)=\frac{1}{2\pi}\Bigl(\log\frac{1+z}{1-z}+i\log\frac{1+iz}{1-iz}\Bigr), \qquad 
z\in\Omega
\end{equation}
and
\begin{equation}\label{power1}
f(z)=\frac{2}{\pi}\sum_{n=0}^\infty \frac{z^{4n+3}}{4n+3},\qquad |z|<1.
\end{equation}
\end{prop}

\begin{proof}
We may write 
\begin{equation}\label{ps f}
f(z)=\frac{2}{\pi}\int_0^z \frac{t^2dt}{1-t^4},\qquad z\in\Omega.
\end{equation}
It is clear that this defines an analytic function in $\Omega$. We may integrate 
along the segment joining $0$ to $z\in\Omega$, which by the star-shaped property of $\Omega$
is contained in $\Omega$ where the integrand $t^2(1-t^4)^{-1}$ is analytic.

When $|z|<1$ we may integrate the Taylor expansion
\[\frac{t^2}{1-t^4}=\sum_{n=0}^\infty t^{4n+2}\]
which proves \eqref{power1}.

The expression
$\frac{1+z}{1-z}$ is a negative real number just in case $z$ is real and $|z|>1$. 
Therefore, $\log \frac{1+z}{1-z}$ is well defined and analytic in $\Omega$. In the 
same way we show that $\log \frac{1+iz}{1-iz}$ is  well defined and analytic in 
$\Omega$. Therefore, the right hand side of \eqref{logform} is an analytic function 
in $\Omega$. Expanding in power series we see that the Taylor series of this 
right hand side function coincides with the power series in \eqref{power1}. 

From this it is clear that we have equality in \eqref{logform} for $|z|<1$, and by analytic continuation we have equality for all $z\in\Omega$.
\end{proof}

\begin{prop}
We have
\begin{equation}\label{firstvalue}
I:= \int_0^1\arctan\Bigl(\frac{\arctanh x-\arctan x}{\pi+\arctanh x-\arctan x}\Bigr)
\frac{dx}{x}=\Im\int_0^1 \log\bigl(1+(1+i)f(z)\bigr)\frac{dz}{z}.
\end{equation}
\end{prop}

\begin{proof}
For all positive real values $x$ we have
\[\arctan\Bigl(\frac{x}{1+x}\Bigr)=\Im \log\bigl(1+(1+i)x\bigr).\]
Therefore,
\[I=\int_0^1 \Im \log\bigl(1+(1+i)f(x)\bigr)\frac{dx}{x}.\]
When $x=0$ we have $f(x)=0$, so that the above logarithm vanishes there. 
For $x$ near $1$ we have $|f(x)|\le C \log(1-x)^{-1}$ by \eqref{logform},
so that the integrand is  $\Orden(\log\log(1-x)^{-1})$. 
Therefore,
the integral 
\[\int_0^1 \log\bigl(1+(1+i)f(x)\bigr)\frac{dx}{x}\]
is well defined, completing the proof.
\end{proof}

\section{Application of Cauchy's Theorem.}

\begin{prop}\label{analytic}
The function
\begin{equation}\label{G}
G(z):=\frac{1}{z}\log\bigl(1+(1+i)f(z)\bigr)
\end{equation}
is an analytic function in the first quadrant.
\end{prop}
\begin{proof}
We will show that $\Re\bigl(1+(1+i)f(z)\bigr)>0$ when $z$ is in the first quadrant. 
Its $\log $ is well defined
and by composition of analytic functions it will be analytic.

The bilinear function $w=\frac{1+z}{1-z}$ transforms the first quadrant into
the points $w$ with $\Im(w)>0$ and $|w|>1$. Then $\log\frac{1+z}{1-z}=a+i\varphi$
where $a>0$ and $0<\varphi<\pi$.  The bilinear transform $w=\frac{1+iz}{1-iz}$ transforms the first quadrant in 
the points $w$ with $\Im(w)>0$ and $|w|<1$. So $\log\frac{1+iz}{1-iz}=-b+i\theta$, 
where $b>0$ and $0<\theta<\pi$. 
Then, for $z$ in the first quadrant, by \eqref{logform} 
\[f(z)=\frac{1}{2\pi}(a+i\varphi-\theta-ib),\]
\[\Re(1+(1+i)f(z))=1+\frac{a+b}{2\pi}-\frac{\varphi+\theta}{2\pi}>0.\]
It is  also clear that at $z=0$ the function $G(z)$ is analytic, because $f(z)$ is 
analytic at $z=0$ and has a zero there.
\end{proof}

The function $G(z)$ in Proposition \ref{analytic} has singularities of ramification
at $z=1$ and $z=i$, but has well defined limits at all other point of the boundary
(in fact it extends analytically at these points). This follows
from the fact that  for $x$ real we have $f(x)$  real  and $f(ix)$  purely imaginary,
so that $1+(1+i)f(z)\ne0$ on the boundary of the first quadrant.

\begin{prop}
For $|z|>1$ in the first quadrant we have
\begin{equation}\label{power2}
f(z)=\frac{i-1}{2}+\frac{2}{\pi}\sum_{n=0}^\infty \frac{1}{(4n+1)z^{4n+1}}.
\end{equation}
\end{prop}
\begin{proof}
Take $z=(1+i)x$ with $x>1$ very large. Then by \eqref{logform}
\[f(z)=\frac{1}{2\pi}\Bigl(\log\frac{z^{-1}+1}{z^{-1}-1}+i\log \frac{(iz)^{-1}+1}{
(iz)^{-1}-1}\Bigr)\]
so that we have, for $x\to\infty$, 
\[\frac{z^{-1}+1}{z^{-1}-1}=-1+\frac{i-1}{x}+\Orden(x^{-2}),\qquad 
\frac{(iz)^{-1}+1}{(iz)^{-1}-1}=-1+\frac{1+i}{x}+\Orden(x^{-2}).\]
It follows that for the main branch of $\log$ we have
\[\log \frac{z^{-1}+1}{z^{-1}-1}=\pi i +\log \frac{1+z^{-1}}{1-z^{-1}},\quad
\log \frac{(iz)^{-1}+1}{(iz)^{-1}-1}=\pi i +\log \frac{1+(iz)^{-1}}{1-(iz)^{-1}}.\]
Therefore, for these values of $z$ we have
\[f(z)=\frac{i-1}{2}+
\frac{1}{2\pi}\Bigl(\log\frac{1+z^{-1}}{1-z^{-1}}+i\log \frac{1+(iz)^{-1}}{1-(iz)^{-1}}
\Bigr)\]
so that we only need to use the known expansion in Taylor series to get the equality 
for $z=(1+i)x$ with $x>1$.  Since $f(z)$ and the expansion are both analytic 
for $|z|>1$ on the first quadrant, we get the equality for $z$ in the first quadrant
with $|z|>1$. 
\end{proof}

\begin{thm}
We have 
\begin{equation}
\int_0^1\arctan\Bigl(\frac{\arctanh x-\arctan x}{\pi+\arctanh x-\arctan x}\Bigr)
\frac{dx}{x}= \frac{\pi}{8}\log\frac{\pi^2}{8}.
\end{equation}
\end{thm}
\begin{proof}
We apply Cauchy's Theorem to $G(z)$ 
\[\int_{C_{R,\varepsilon}}G(z)\,dz=0.\]
Here $C_{R,\varepsilon}$ is the path shown in Figure \ref{fig}. We avoid the singularities
by small semicircles of radius $\varepsilon>0$, where $0<\varepsilon<1<R$.
\begin{figure}[H]
  \includegraphics[width=6.8cm]{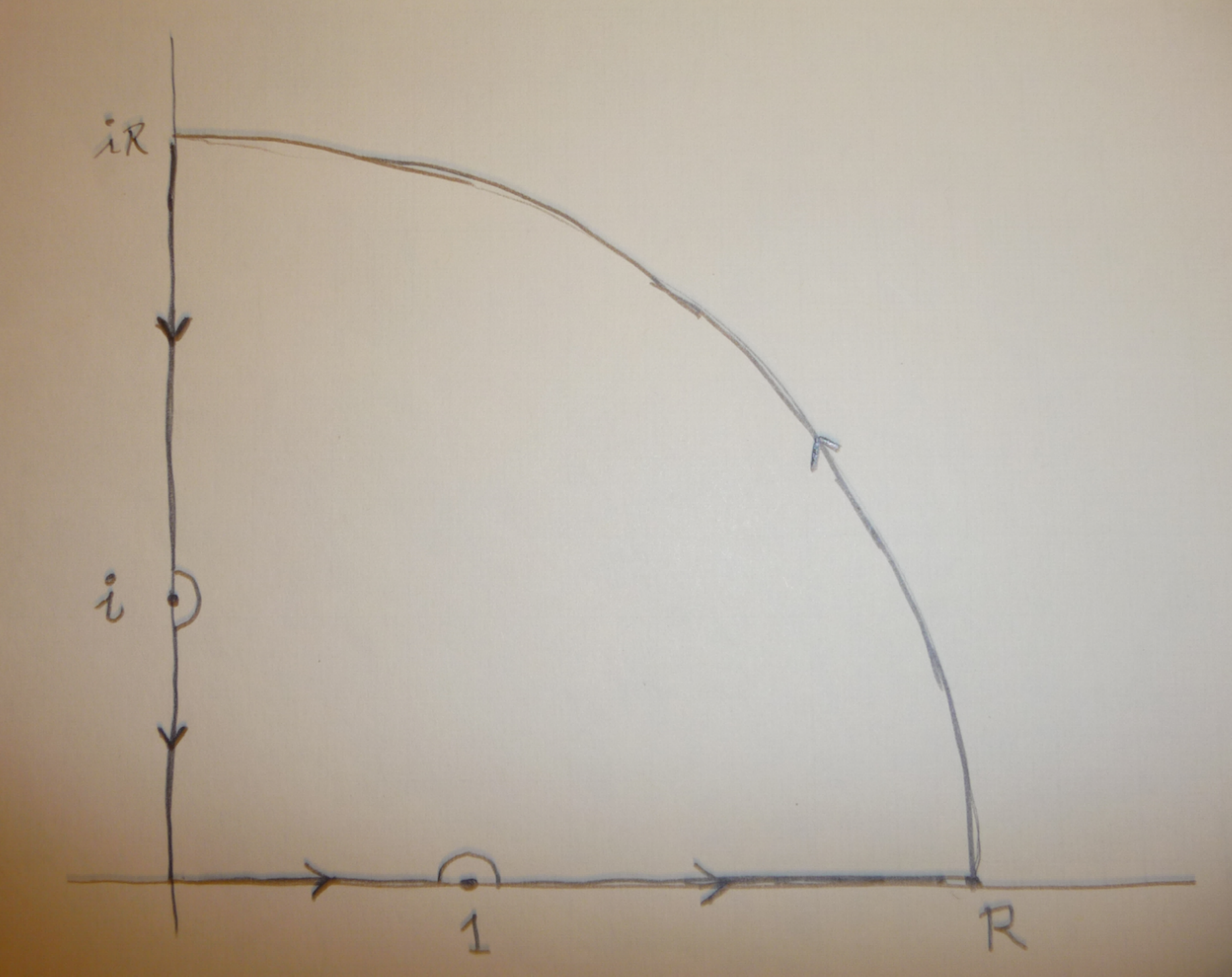}
  \caption{Path of integration.}\label{fig}
\end{figure}

The integrals along the small semicircles tend to $0$ when $\varepsilon\to 0$,
so that  by Cauchy's Theorem 
\begin{equation}\label{Cauchy1}
\int_0^1 G(x)\,dx+\int_1^R G(x)\,dx+iR\int_0^{\pi/2} G(R e^{ix}) e^{ix}\,dx
-\int_0^i G(z)\,dz-i\int_1 ^{R} G(iy)\,dy=0.
\end{equation}
Two of the integrals combine to give our integral. In fact
\[\int_0^1 G(x)\,dx-\int_0^i G(z)\,dz=\int_0^1 G(x)\,dx-i\int_0^1 G(ix)\,dx,\]
and we have 
\[G(x)-iG(x)=\frac{1}{x}\log\bigl(1+(1+i)f(x)\bigr)-\frac{i}{ix}
\log\bigl(1+(1+i)f(ix)\bigr).\]
By \eqref{power1} for $0<x<1$ we have $f(ix)=-if(x)$, so that 
\begin{multline*}
G(x)-iG(x)=\frac{\log\bigl(1+(1+i)f(x)\bigr)-
\log\bigl(1+(1-i)f(x)\bigr)}{x}\\
=\frac{2i}{x}\Im\log\bigl(1+(1+i)f(x)\bigr).
\end{multline*}
Then our two integrals are
\[\int_0^1 G(x)\,dx-\int_0^i G(z)\,dz=2i\int_0^1\Im\log\bigl(1+(1+i)f(x)\bigr)
\frac{dx}{x}.\]
Applying  \eqref{firstvalue} we get 
\begin{equation}
\int_0^1 G(x)\,dx-\int_0^i G(z)\,dz=2Ii.
\end{equation}
Our equation \eqref{Cauchy1}, obtained by Cauchy's Theorem, may now be written as
\begin{equation}\label{Cauchy2}
2iI+\int_1^R G(x)\,dx+iR\int_0^{2\pi} G(R e^{ix}) e^{ix}\,dx
-i\int_1 ^{R} G(iy)\,dy=0.
\end{equation}
In a similar way we also find that
\[\int_1^R G(x)\,dx-i\int_1 ^{R} G(iy)\,dy=\int_1^R(G(x)-iG(ix))\,dx\]
with
\[G(x)-iG(ix)=\frac{1}{x}\bigl\{\log\bigl(1+(1+i)f(x)\bigr)-
\log\bigl(1+(1+i)f(ix)\bigr)\bigr\}.\]
Here $x>1$ and we substitute the values of $f(x)$ and $f(ix)$ given 
by \eqref{power2}
\[1+(1+i)f(x)=\frac{2(1+i)}{\pi}\sum_{n=0}^\infty \frac{1}{(4n+1)x^{4n+1}},\]
\[1+(1+i)f(ix)=\frac{2(1-i)}{\pi}\sum_{n=0}^\infty \frac{1}{(4n+1)x^{4n+1}}.\]
We arrive at
\begin{multline*}
\int_1^R G(x)\,dx-i\int_1 ^{R} G(iy)\,dy\\=
\int_1^R\Bigl\{\log\Bigl(\frac{2(1+i)}{\pi}\sum_{n=0}^\infty \frac{1}{(4n+1)x^{4n+1}}
\Bigr)-\log\Bigl(\frac{2(1-i)}{\pi}\sum_{n=0}^\infty \frac{1}{(4n+1)x^{4n+1}}\Bigr)
\Bigr\}\frac{dx}{x}.
\end{multline*}
Both logarithms have the same real part. Therefore, only the integrals 
of the imaginary parts remain
\[\int_1^R G(x)\,dx-i\int_1 ^{R} G(iy)\,dy=\int_1^R\frac{\pi i}{2}\frac{dx}{x}=\frac{\pi i}{2}\log R.\]
Substituting this in  \eqref{Cauchy2} we get
\begin{equation}\label{Cauchy3}
2iI+\frac{\pi i}{2}\log R+iR\int_0^{\pi/2} G(R e^{ix}) e^{ix}\,dx
=0.
\end{equation}

The last integral can be transformed in the following way
\begin{multline*}
iR\int_0^{\pi/2} G(R e^{ix}) e^{ix}\,dx=iR\int_0^{\pi/2} 
\frac{1}{Re^{ix}}\log\bigl(1+(1+i)f(Re^{ix})\bigr)
e^{ix}\,dx\\
=i\int_0^{\pi/2}\log\bigl(1+(1+i)f(Re^{ix})\bigr)\,dx.
\end{multline*}
Since $R>1$ and $Re^{ix}$ is in the first quadrant,  the function $f(z)$ can be
computed by \eqref{power2}. 
\[1+(1+i)f(Re^{ix})=\frac{2(1+i)}{\pi}\sum_{n=0}^\infty\frac{e^{-ix(4n+1)}}{(4n+1)R^{4n+1}}=
\frac{2(1+i)}{\pi R}e^{-ix}+\frac{2(1+i)}{5\pi R^5}e^{-i5x}+\cdots\]
Then for $R$ large enough 
\[\log\bigl(1+(1+i)f(Re^{ix})\bigr)=\log\Bigl(\frac{2(1+i)}{\pi R}e^{-ix}\Bigr)+
\log\Bigl(1+\frac{e^{-i4x}}{5R^4}+\frac{e^{-i8x}}{9R^8}+\cdots\Bigr)\]
and
\[i\int_0^{\pi/2}\log\bigl(1+(1+i)f(Re^{ix})\bigr)\,dx=
i\int_0^{\pi/2}\Bigl\{\log\frac{2\sqrt{2}}{\pi R}+\Bigl(\frac{\pi}{4}-x\Bigr)i
\Bigr\}\,dx +\Orden(R^{-4})\]
or\footnote{The term $\Orden(R^{-4})$ is  $=0$ but we do not need to prove this.} 
\[i\int_0^{\pi/2}\log\bigl(1+(1+i)f(Re^{ix})\bigr)\,dx=-\frac{\pi i}{2}\log R+
\frac{\pi i}{2}\log\frac{2\sqrt{2}}{\pi }+\Orden(R^{-4}).\]
Substituting this in \eqref{Cauchy3} we get
\begin{equation}\label{Cauchy4}
2iI+
\frac{\pi i}{2}\log\frac{2\sqrt{2}}{\pi }+\Orden(R^{-4})
=0.
\end{equation}
Taking limits for $R\to\infty$ we get 
\[I=-\frac{\pi}{4}\log\frac{2\sqrt{2}}{\pi }=\frac{\pi}{8}\log\frac{\pi^2}{8}.\]
\end{proof}
\section{Acknowledgement}

I wish to express my thanks to 
Jan van de Lune  ( Hallum, The Netherlands ) for his encouragements and linguistic assistance.


\begin{thebibliography}{99}

\bibitem{exp}
D. H. Bailey, J. M. Borwein, N. J. Calkin, R. Girgensohn, D. R. Luke, V. H. Moll, 
\emph{Experimental Mathematics in Action}, A. K. Peters, 2007. 
\bigskip

\bibitem{larry}
MathStackExchange user larry, 
\url{http://math.stackexchange.com/questions/464769}
\bigskip


\bibitem{zy}
MathOverflow user zy\_, 
\url{http://mathoverflow.net/questions/154913}
\bigskip


\bibitem{Reshetnikov}
V. Reshetnikov, \url{http://math.stackexchange.com/questions/521993}


\end{thebibliography}
\end{document}